\documentclass[letterpaper,12pt,reqno, english]{amsart}  
\usepackage[utf8]{inputenc}
\usepackage[T1]{fontenc}
\usepackage{amsmath,amsthm}
\usepackage{amsfonts,amssymb,enumerate}
\usepackage{url,paralist}
\usepackage[colorlinks=true,urlcolor=blue,linkcolor=red,citecolor=magenta]{hyperref}
\usepackage{enumerate}
\usepackage{anysize}
\usepackage{tikz}

\theoremstyle{plain}

\newtheorem{theorem}{Theorem}[section]

\newtheorem{lemma}[theorem]{Lemma}

\theoremstyle{definition}

\theoremstyle{remark}

\newtheorem{remark}[theorem]{Remark}

\newcommand{\F}{\mathcal{F}}
\newcommand{\I}{\mathcal{I}}

\newcommand{\M}{\mathcal{M}}

\newcommand{\R}{\mathbb{R}}

\DeclareMathOperator{\conv}{\mathrm{conv}}
\DeclareMathOperator{\supp}{\mathrm{supp}}

\newcommand{\T}{\mathcal{T}}

\begin{document}

\title[Colorful coverings of polytopes]{Colorful coverings of polytopes and piercing numbers of colorful $d$-intervals}

\author{Florian Frick}\thanks{}
\address[FF]{Dept.\ Math., Cornell University, Ithaca, NY}
\email{ff238@cornell.edu}
\author{Shira Zerbib}
\address[SZ]{Dept.\ Math., University of Michigan, Ann Arbor, MI}
\email{zerbib@umich.edu}


\begin{abstract}
	\small
	We prove a common strengthening of B\'ar\'any's colorful Carath\'eodory theorem and the KKMS theorem.
	In fact, our main result is a colorful polytopal KKMS theorem, which extends a colorful KKMS theorem due to 
	Shih and Lee~[\emph{Math. Ann.} 296 (1993), no. 1, 35--61] as well as a polytopal KKMS theorem due to 
	Komiya~[\emph{Econ. Theory} 4 (1994), no. 3, 463--466]. The (seemingly unrelated) colorful Carath\'eodory 
	theorem is a special case as well. We apply our theorem to establish an upper bound on the piercing number
	of colorful $d$-interval hypergraphs, extending earlier results of Tardos~[\emph{Combinatorica} 15 (1995), no. 1, 123--134] 
	and Kaiser~[\emph{Discrete Comput. Geom.} 18 (1997), no. 2, 195--203].
\end{abstract}

\maketitle


\bigskip


\section{Introduction}
\label{sec.int}

The KKM theorem of Knaster, Kuratowski, and Mazurkiewicz~\cite{KKM} is a set covering variant of Brouwer's fixed point theorem.
It states that for any covering of the $k$-simplex $\Delta_k$ on vertex set $[k+1]$ with closed sets $A_1, \dots, A_{k+1}$ such that
the face spanned by vertices in $S$ is contained in $\bigcup_{i\in S} A_i$ for every $S \subset [k+1]$, the intersection $\bigcap_{i \in [k+1]} A_i$
is nonempty.

The KKM theorem has inspired many extensions and variants, some of which we will briefly survey in Section~\ref{sec:komiya}. Important strengthenings
include a colorful extension of the KKM theorem due to Gale~\cite{gale1984} that deals with $k+1$ possibly distinct coverings of the $k$-simplex and the 
KKMS theorem of Shapley~\cite{shapley}, where the sets in the covering are associated to faces of the $k$-simplex instead of its vertices. Further generalizations
of the KKMS theorem are a polytopal version due to Komiya~\cite{komiya} and the colorful KKMS theorem of Shih and Lee~\cite{ShihLee}.

In this note we prove a colorful polytopal KKMS theorem, extending all results above. This result is finally sufficiently general to also 
specialize to B\'ar\'any's celebrated colorful Carath\'eodory theorem~\cite{barany} from 1982, which asserts that if 
$X_1, \dots, X_{k+1}$ are subsets of $\R^k$ with $0 \in \conv X_i$ for every $i \in [k+1]$, then there exists a choice of points  
$x_1 \in X_1, \dots, x_{k+1} \in X_{k+1}$ such that $0 \in \conv\{x_1, \dots, x_{k+1}\}$.  
Carath\'eodory's classical result is the case $X_1 = X_2 = \dots = X_{k+1}$. We deduce the colorful Carath\'eodory theorem
from our main result in Section~\ref{sec5}.

For a set $\sigma \subset \R^k$ we denote by $C_\sigma$ the \emph{cone of $\sigma$}, that is, the union of all rays emanating 
from the origin that intersect~$\sigma$. Our main result is the following: 
 
\begin{theorem}
\label{thm:col-komiya}
	Let $P$ be a $k$-dimensional polytope with~${0 \in P}$. Suppose for every nonempty, proper face $\sigma$ of $P$ we are given $k+1$ points
	$y^{(1)}_\sigma, \dots, y^{(k+1)}_\sigma \in C_\sigma$ and $k+1$ closed sets $A^{(1)}_\sigma, \dots, A^{(k+1)}_\sigma \subset P$.
	If $\sigma \subset \bigcup_{\tau \subset \sigma} A^{(j)}_\tau$ for every face~$\sigma$ of~$P$ and every~${j\in [k+1]}$,
	then there exist faces $\sigma_1, \dots, \sigma_{k+1}$ of $P$ such that
	$0 \in \conv\{y_{\sigma_1}^{(1)}, \dots, y_{\sigma_{k+1}}^{(k+1)}\}$ and $\bigcap_{i=1}^{k+1} A^{(i)}_{\sigma_i}\neq \emptyset.$
\end{theorem}

Our proof of this result relies on a 
topological mapping degree argument. As such, it is entirely different from B\'ar\'any's proof of the colorful Carath\'eodory theorem, and thus provides a new topological route to prove this theorem. Our argument is also less involved
than the topological proof given recently by Meunier, Mulzer, Sarrabezolles, and Stein~\cite{meunier2017} to show that algorithmically  
finding the configuration whose existence is guaranteed by the colorful Carath\'eodory theorem is in PPAD (that is, informally
speaking, it can be found by a path-following algorithm). Our method, however, involves a limiting argument and thus does not have immediate
algorithmic consequences. 
Finally, our proof of Theorem \ref{thm:col-komiya} exhibits a surprisingly simple way to prove KKMS-type results and their polytopal and colorful extensions.

As an application of Theorem~\ref{thm:col-komiya} we prove a bound on the piercing numbers of colorful $d$-interval hypergraphs.  
A {\em $d$-interval} is a union of at most $d$ disjoint closed
intervals on~$\mathbb{R}$. A $d$-interval $h$ is {\em separated} if it consists of $d$ disjoint interval components $h = h^1 \cup \dots \cup h^d$ with $h^{i+1} \subset (i, i + 1)$ for $i \in\{0, \dots, d-1\}$.  
A {\em hypergraph of (separated) $d$-intervals} is a hypergraph $H$ whose vertex set is $\mathbb{R}$ and whose edge set is a finite family of (separated) $d$-intervals.

A {\em matching} in a hypergraph $H=(V,E)$ with vertex set $V$ and
edge set $E$ is a set of disjoint edges. A {\em cover} is a subset of
$V$ intersecting all edges. The \emph{matching number} $\nu(H)$ is the maximal size of a matching, and
the \emph{covering number} (or {\em piercing number}) $\tau(H)$ is the minimal size of a
cover. 
Tardos \cite{tardos} and Kaiser \cite{kaiser} proved the following bound on the covering number in
hypergraphs of $d$-intervals:

\begin{theorem}[Tardos \cite{tardos}, Kaiser \cite{kaiser}]\label{t:kaiser}
In every hypergraph $H$ of
  $d$-intervals we have 
$\tau(H) \leq (d^2-d+1)\nu(H).$ Moreover, if $H$ is a hypergraph of separated $d$-intervals then $  \tau(H) \leq (d^2-d)\nu(H).$
  \end{theorem}

Matou\v{s}ek~\cite{matousek} showed that this bound is not far from
the truth: There are examples of hypergraphs of
$d$-intervals in which $\tau = \Omega(\frac{d^2}{\log
  d}\nu)$. Aharoni, Kaiser and Zerbib~\cite{AKZ} gave a proof of Theorem~\ref{t:kaiser} that used the KKMS theorem and Komiya's polytopal extension, Theorem~\ref{thm:komiya}. 
Using Theorem \ref{thm:col-komiya} we prove here a colorful generalization of Theorem~\ref{t:kaiser}:

\begin{theorem}\label{coloreddintervals}
\begin{enumerate}[1.] 
\item   Let $\F_i,~ i\in[k+1]$, be $k+1$ hypergraphs of $d$-intervals and let $\F= \bigcup_{i=1}^{k+1} \F_i$.
If $\tau(\F_i )>k$ for all $i\in[k+1]$, then there exists a collection $\M$  of pairwise disjoint $d$-intervals in $\F$ of size $|\M| \ge \frac{k+1}{d^2-d+1}$, with $|\M\cap \F_i| \le 1$. 
\item Let $\F_i,~ i\in[kd+1]$, be $kd+1$ hypergraphs of separated $d$-intervals and let $\F= \bigcup_{i=1}^{k+1} \F_i$.
If $\tau(\F_i )>kd$ for all $i \in [k+1]$, then there exists a collection $\M$  of pairwise disjoint separated $d$-intervals in $\F$ of size $|\M| \ge \frac{k+1}{d-1}$, with $|\M\cap \F_i | \le 1$. 
\end{enumerate}      
\end{theorem}

Note that Theorem~\ref{t:kaiser} is the case where all the hypergraphs $\F_i $ are the same. 
In Section~\ref{sec:komiya} we introduce some notation and, as an introduction to our methods, provide a new simple proof of Komiya's theorem. Then, in Section~\ref{sec5}, we prove Theorem~\ref{thm:col-komiya} and use it to derive B\'ar\'any's colorful Carath\'eodory theorem. Section~\ref{sec:interval} 
is devoted to the proof of Theorem~\ref{coloreddintervals}.

\section{Coverings of polytopes and Komiya's theorem}
\label{sec:komiya}

Let $\Delta_k$ be the $k$-dimensional simplex with vertex set $[k+1]$ realized in $\R^{k+1}$ as 
$\{x \in \R^{k+1}_{\ge 0} \: : \: \sum_{i=1}^{k+1} x_i = 0\}$. For every $S\subset[k+1]$ let $\Delta^S$ be the face of $\Delta_k$ spanned by the vertices in~$S$.
Recall that the KKM theorem asserts that if $A_1,\dots,A_{k+1}$ are closed sets covering $\Delta_k$ so that 
$\Delta^S \subset \bigcup_{i\in S} A_i$ for every $S\subset [k+1]$, then the intersection of all the sets $A_i$ is non-empty.
We will refer to covers $A_1, \dots, A_{k+1}$ as above as \emph{KKM cover}.

A generalization of this result, known as the KKMS theorem, was proven by Shapley~\cite{shapley} in 1973. Now we have a cover of $\Delta_k$ by closed sets $A_T,~T\subset [k+1]$, so that $\Delta^S \subset \bigcup_{T\subset S} A_T$ for every $S\subset [k+1]$. Such a collection of sets $A_T$ is called {\em KKMS cover}. The conclusion of the KKMS theorem is that there exists a balanced collection of $T_1,\dots,T_m$ of subsets of $[k+1]$ for which $\bigcap_{i=1}^{m} A_{T_i} \neq \emptyset$. Here $T_1,\dots,T_m$ form a balanced collection if the barycenters of the corresponding faces $\Delta_{T_1},\dots,\Delta_{T_m}$ contain the barycenter of $\Delta_k$ in their convex hull.  

A different generalization of the KKM theorem is a colorful version due to Gale~\cite{gale1984}. It states that given $k+1$ KKM covers
$A_1^{(i)}, \dots, A_{k+1}^{(i)}$, $i \in [k+1]$, of the $k$-simplex~$\Delta_k$, there is a permutation $\pi$ of $[k+1]$
such that $\bigcap_{i \in [k+1]} A_{\pi(i)}^{(i)}$ is nonempty. This theorem is colorful in the sense that we think of each 
KKM cover as having a different color; the theorem then asserts that there is an intersection of $k+1$ sets of pairwise
distinct colors associated to pairwise distinct vertices. Asada et al.~\cite{asada2017} showed that one can additionally prescribe~$\pi(1)$.

In 1993 Shih and Lee~\cite{ShihLee} proved a common generalization of the KKMS theorem and Gale's colorful KKM theorem: Given $k+1$ such KKMS covers $A_T^{i},~T\subset [k+1],~i\in [k+1]$, of $\Delta_k$, there exists a balanced collection of $T_1,\dots,T_{k+1}$ of subsets of $[k+1]$ for which we have $\bigcap_{i=1}^{m} A_{T_i}^i \ne \emptyset$.

Another far reaching extension of the KKMS theorem to general polytopes is due to Komiya~\cite{komiya} from 1994. Komiya proved that the simplex $\Delta_k$ in the KKMS theorem can be replaced by any $k$-dimensional polytope $P$, and that the barycenters of the faces can be replaced by any points $y_{\sigma}$ in the face~$\sigma$:

\begin{theorem}[Komiya's theorem~\cite{komiya}]
\label{thm:komiya}
	Let $P$ be a polytope, and for every nonempty face $\sigma$ of $P$ choose a point $y_\sigma \in \sigma$
	and a closed set $A_\sigma \subset P$. If $\sigma \subset \bigcup_{\tau \subset \sigma} A_\tau$
	for every face~$\sigma$ of~$P$,  then there are faces $\sigma_1, \dots, \sigma_m$ of $P$ such that
	$y_P \in \conv\{y_{\sigma_1}, \dots, y_{\sigma_m}\}$ and $\bigcap_{i=1}^m A_{\sigma_i} \neq \emptyset$. 
\end{theorem}

This specializes to the KKMS theorem if $P$ is the simplex and each point $y_\sigma$ is the barycenter
of the face~$\sigma$. Moreover, there are quantitative versions of the KKM theorem due to De Loera, Peterson, and Su~\cite{deLoera2002}
as well as Asada et al.~\cite{asada2017} and KKM theorems for general pairs of spaces due to Musin~\cite{musin2017}.

To set the stage we will first present a simple proof of Komiya's theorem. Recall that 
the KKM theorem can be easily deduced from Sperner's lemma on vertex labelings of triangulations of a simplex.
Our proof of Komiya's theorem -- just as Shapley's original proof of the KKMS theorem -- first establishes an equivalent Sperner-type 
version. A \emph{Sperner--Shapley labeling} of a triangulation $T$ of a polytope $P$ is a map 
$f\colon V(T) \longrightarrow \{\sigma \: : \: \sigma \ \text{a nonempty face of} \ P\}$ from the vertex set $V(T)$ of $T$ to the set of nonempty
faces of $P$ such that $f(v) \subset \supp(v)$, where $\supp(v)$ is the minimal face of $P$ containing $v$.
We prove the following polytopal Sperner--Shapley theorem that will imply Theorem~\ref{thm:komiya} by a limiting and 
compactness argument:

\begin{theorem}
\label{thm:sperner-s}
	Let $T$ be a triangulation of the polytope~$P \subset \R^k$, and let $f\colon V(T) \longrightarrow \{\sigma \: : \: \sigma \ \text{a nonempty face of} \ P\}$
	be a Sperner--Shapley labeling of $T$. For every nonempty face $\sigma$ of $P$ choose a point $y_\sigma \in \sigma$.
	Then there is a face $\tau$ of $T$ such that $y_P \in \conv\{y_{f(v)} \: : \: v \ \text{vertex of} \ \tau\}$.
\end{theorem}

\begin{proof}
	The Sperner--Shapley labeling $f$ maps vertices of 
	the triangulation $T$ of $P$ to faces of~$P$; thus mapping vertex $v$ to the chosen point $y_{f(v)}$ in the face~$f(v)$ 
	and extending linearly onto faces of~$T$ defines a continuous map $F\colon P \longrightarrow P$. By the Sperner--Shapley 
	condition for every face $\sigma$ of $P$ we have that $F(\sigma) \subset \sigma$. This implies that $F$ is 
	homotopic to the identity on~$\partial P$, and thus $F|_{\partial P}$ has degree one. Then $F$ is surjective 
	and we can find a point $x \in P$ such that $F(x) = y_P$. Let $\tau$ be the smallest face of $T$ containing~$x$.
	By definition of $F$ the image $F(\tau)$ is equal to the convex hull $\conv\{y_{f(v)} \: : \: v \ \text{vertex of} \ \tau\}$.
\end{proof}

\begin{proof}[Proof of Theorem~\ref{thm:komiya}]
	Let $\varepsilon > 0$, and let $T$ be a triangulation of~$P$ such that every face of $T$ has diameter at most~$\varepsilon$. 
	Given a cover $\{A_\sigma \: : \: \sigma \ \text{a nonempty face of} \ P\}$ that satisfies the covering condition of the theorem 
	we define a Sperner--Shapley labeling in the following way: 
	For a vertex $v$ of $T$, label $v$ by a face $\sigma \subset \supp(v)$ such that $v\in A_{\sigma}$. 
	Such a face $\sigma$ exists since $v\in\supp(v) \subset \bigcup_{\sigma \subset \supp(v)} A_{\sigma}$. Thus by 
	Theorem~\ref{thm:sperner-s} there is a face $\tau$ of $T$ whose vertices are labeled by faces 
	$\sigma_1, \dots, \sigma_m$ of $P$ such that $y_P \in \conv\{y_{\sigma_1}, \dots, y_{\sigma_m}\}$. 
	In particular, the $\varepsilon$-neighborhoods of the sets $A_{\sigma_i}$, $i \in [m]$, intersect.
	Now let $\varepsilon$ tend to zero. As there
	are only finitely many collections of faces of~$P$, one collection $\sigma_1, \dots, \sigma_m$ must appear infinitely
	many times. By compactness of $P$ the sets $A_{\sigma_i}$, $i \in [m]$, then all intersect since they are closed.
\end{proof}

Note that Theorem \ref{thm:komiya} is true also if all the sets $A_{\sigma}$ are open in~$P$. 
Indeed, given an open cover $\{A_\sigma \: : \: \sigma \ \text{a nonempty face of} \ P\}$ of $P$ as in Theorem~\ref{thm:komiya},
we can find closed sets $B_\sigma \subset A_\sigma$ that have the same nerve as $A_\sigma$ (namely, any collection of sets $\{B_{\sigma_i}\: : \:  i\in I\}$ intersects if and only if  the corresponding collection $\{A_{\sigma_i}\: : \: i\in I\}$ intersects) and
still satisfy $\sigma \subset \bigcup_{\tau \subset \sigma} B_\tau$ for every face~$\sigma$ of~$P$.

\section{A colorful Komiya theorem}\label{sec5}

Recall that the colorful KKMS theorem of Shih and Lee~\cite{ShihLee} states the following: If for every $i\in [k+1]$ the collection $\{A^{i}_{\sigma} \: : \: \sigma \text{ a nonempty face of }\Delta_k\}$ forms a KKMS cover of $\Delta_k$, then there exists a balanced collection of faces $\sigma_1,\dots,\sigma_{k+1}$ so that $\bigcap_{i=1}^{k+1}A_{\sigma_i}^i \neq \emptyset$. Theorem~\ref{thm:col-komiya}, proved in this section, is a colorful extension of Theorem~\ref{thm:komiya}, and thus generalizes the colorful KKMS theorem to any polytope. 
  
Let $P$ be a $k$-dimensional polytope. Suppose that for every nonempty face $\sigma$ of $P$ we choose $k+1$ points
$y^{(1)}_\sigma, \dots, y^{(k+1)}_\sigma \in \sigma$ and $k+1$ closed sets $A^{(1)}_\sigma, \dots, A^{(k+1)}_\sigma \subset P$, so that $\sigma \subset \bigcup_{\tau \subset \sigma} A^{(j)}_\tau$ for every face~$\sigma$ of~$P$ and every~${j\in [k+1]}$.
Theorem~\ref{thm:komiya} now guarantees that for every fixed $j \in [k+1]$ there are faces $\sigma_1^{(j)},\dots, \sigma_{m_j}^{(j)}$ of $P$ such that
$y_P^{(j)} \in \conv\{y_{\sigma_1}^{(j)}, \dots, y_{\sigma_{m_j}}^{(j)}\}$ and $\bigcap_{i=1}^{m_j} A_{\sigma_i}^{(j)}$ is nonempty. 
Now let us choose $y_P^{(1)} = y_P^{(2)} = \dots = y_P^{(k+1)}$ and denote this point by~$y_P$.
The colorful Carath\'eodory theorem implies the existence of points $z_j \in \{y_{\sigma_1}^{(j)}, \dots, y_{\sigma_{m_j}}^{(j)}\}$,
$j \in [k+1]$, such that $y_P \in \conv\{z_1, \dots, z_{k+1}\}$. Theorem \ref{thm:col-komiya} shows that this implication can be realized simultaneously with the existence of sets
$B_j \in \{A_{\sigma_1}^{(j)}, \dots, A_{\sigma_{m_j}}^{(j)}\}$, $j \in [k+1]$, such that $\bigcap_{j=1}^{k+1} B_j$ is nonempty. 
We prove Theorem~\ref{thm:col-komiya} by applying the Sperner--Shapley version of Komiya's theorem -- Theorem~\ref{thm:sperner-s} -- to
a labeling of the barycentric subdivision of a triangulation of~$P$. The same idea was used by Su~\cite{su1999} to prove a colorful Sperner's lemma.
For related Sperner-type results for multiple Sperner labelings see Babson~\cite{babson2012}, Bapat~\cite{bapat1989}, 
and Frick, Houston-Edwards, and Meunier~\cite{frick2017}.

\begin{proof}[Proof of Theorem~\ref{thm:col-komiya}]
	Let $\varepsilon > 0$, and let $T$ be a triangulation of~$P$ such that every face of $T$ has diameter at most~$\varepsilon$. 
We will also assume that the chosen points $y^{(1)}_\sigma, \dots, y^{(k+1)}_\sigma$ are contained in~$\sigma$. This assumption does not restrict the generality of our proof 
since $0 \in \conv\{x_1, \dots, x_{k+1}\}$ for vectors $x_1, \dots, x_{k+1} \in \R^k$ if and only if
	$0 \in \conv\{\alpha_1x_1, \dots, \alpha_{k+1}x_{k+1}\}$ with arbitrary coefficients $\alpha_i > 0$.
	Denote by $T'$ the barycentric subdivision of~$T$. We now define a Sperner--Shapley labeling of the vertices of $T'$: 
For	$v\in V(T')$ let $\sigma_v$ be the face of $T$ so that $v$ lies at the barycenter of $\sigma_v$, let $\ell=\dim \sigma_v$, and let $\sigma$ be the minimal supporting face of $P$ containing  $\sigma_v$. By the conditions of the theorem, $v$ is contained in a set $A^{(\ell+1)}_\tau$ where $\tau \subset \sigma$. 
	We label $v$ by~$\tau$. Thus by Theorem~\ref{thm:sperner-s} there exists a face $\tau$ of $T'$ (without loss of generality $\tau$ is a facet) 
	whose vertices are labeled by faces $\sigma_1, \dots, \sigma_{k+1}$ of $P$ such that 
	$0 \in \conv\{y^{(1)}_{\sigma_1}, \dots, y^{(k+1)}_{\sigma_{k+1}}\}$.  In particular, the $\varepsilon$-neighborhoods of the 
	sets $A^{(i)}_{\sigma_i}$, $i \in [k+1]$, intersect. Now use a limiting argument as before.
\end{proof}

Note that by the same argument as before, Theorem \ref{thm:col-komiya} is true also if all the sets $A^{(i)}_\sigma$ are open. 

For a point $x\neq 0$ in $\R^k$ let $H(x) = \{y \in \R^k \: : \: \langle x,y \rangle = 0\}$ be the hyperplane perpendicular to $x$ and let $H^+(x) = \{y \in \R^k \: : \: \langle x,y \rangle \ge 0\}$ be the closed halfspace with boundary $H(x)$ containing~$x$. 
Let us now show that B\'ar\'any's colorful Carath\'eodory theorem is a special case of Theorem~\ref{thm:col-komiya}. 

\begin{theorem}[Colorful Carath\'eodory theorem, B\'ar\'any~\cite{barany}]
\label{thm:col-car}
	Let $X_1, \dots, X_{k+1}$ be finite subsets of $\R^k$ with $0 \in \conv X_i$ for every $i \in [k+1]$. Then there are
	$x_1 \in X_1, \dots, x_{k+1} \in X_{k+1}$ such that $0 \in \conv\{x_1, \dots, x_{k+1}\}$.
\end{theorem}

\begin{proof}
We will assume that $0$ is not contained in any of the sets $X_1,...,X_{k+1}$, for otherwise we are done.
Let $P \subset \R^k$ be a polytope containing $0$ in its interior, such that if points $x$ and $y$ belong to the same face of $P$ then ${\langle x,y \rangle \ge 0}$.
For example, a sufficiently fine subdivision of any polytope that contains 0 in its interior (slightly perturbed to be a strictly convex polytope) satisfies this condition.
We can assume that any ray emanating from the origin intersects each $X_i$ in at most one point by arbitrarily deleting 
any additional points from~$X_i$. This will not affect the property that~${0 \in \conv X_i}$. Furthermore, we can choose $P$ in such a way that 
for each face $\sigma$ and $i \in [k+1]$ the intersection $C_\sigma \cap X_i$ contains at most one point.

For every $i\in [k+1]$ let $y_P^{(i)} = 0$ and $A_P^{(i)} = \emptyset$. Now for each nonempty, proper 
face $\sigma$ of $P$ choose points $y_{\sigma}^{(i)}$ and sets $A_{\sigma}^{(i)}$ in the following way: 
If there exists $x \in C_\sigma \cap X_i$ then let $y_\sigma^{(i)} = x$ and 
$A_\sigma^{(i)} = \{y \in P \: : \: \langle y, x \rangle \ge 0\} = P\cap H^+(x)$; otherwise let $y_\sigma^{(i)}$ be some point in $\sigma$ 
and let~${A_\sigma^{(i)} = \sigma}$. 

Suppose the statement of the theorem was incorrect; then in particular, we can slightly perturb the vertices of $P$ and those points $y_{\sigma}^{(i)}$
that were chosen arbitrarily in~$\sigma$, to make sure that for any collection of points $y_{\sigma_1}^{(1)}, \dots, y_{\sigma_{k+1}}^{(k+1)}$
and any subset $S$ of this collection of size at most $k$, $0\notin \conv S$. 

Let us now check that with these definitions the conditions of Theorem~\ref{thm:col-komiya} hold. Clearly, all the sets $A_\sigma^{(i)}$ are closed. 
The fact that  
$P$ is covered by the sets $A_\sigma^{(i)}$ for every fixed~$i$ follows from the condition $0 \in \conv X_i$. Indeed, this condition implies that for every $p \in P$ there exists a point $x \in X_i$ with $\langle p, x \rangle \ge 0$, and therefore, for the face $\sigma$ of $P$
for which ~${x \in C_\sigma}$ we have ${p \in A_{\sigma}^{(i)}}$.

Now fix a proper face $\sigma$ of~$P$. We claim that $\sigma \subset A_\sigma^{(i)}$ for every~$i$. Indeed, either
$X_i \cap C_\sigma = \emptyset$ in which case $A_\sigma^{(i)} = \sigma$, or otherwise, 
pick $x \in X_i \cap C_\sigma$ and let $\lambda > 0$ such that $\lambda x \in \sigma$; then for every $p \in \sigma$
we have $\langle p, \lambda x \rangle\ge 0$ by our assumption on~$P$, and thus $\langle p, x \rangle \ge 0$, 
or equivalently~${p \in A_\sigma^{(i)}}$. 

Thus by Theorem~\ref{thm:col-komiya} there exist faces $\sigma_1, \dots, \sigma_{k+1}$ of $P$ 
such that $\bigcap_{i=1}^{k+1} A_{\sigma_i}^{(i)} \ne \emptyset$ and $0 \in \conv\{y_{\sigma_1}^{(1)}, \dots, y_{\sigma_{k+1}}^{(k+1)}\}$. 
We claim that $\bigcap_{i=1}^{k+1} A_{\sigma_i}^{(i)}$ can contain only the origin. Indeed, suppose that $0\neq x_0 \in \bigcap_{i=1}^{k+1} A_{\sigma_i}^{(i)}$. Fix $i\in [k+1]$. If $y_{\sigma_i}^{(i)}\in C_{\sigma_i} \cap X_i$, then since $x_0 \in  A_{\sigma_i}^{(i)}$ we have $y_{\sigma_i}^{(i)} \in H^+(x_0)$ by definition.   
Otherwise $x_0 \in A_{\sigma_i}^{(i)} = \sigma_i$ and $y_{\sigma_i}^{(i)}\in \sigma_i$, so by our choice of $P$ we obtain again that $y_{\sigma_i}^{(i)} \in H^+(x_0)$.     
Thus all the points $y_{\sigma_1}^{(1)}, \dots, y_{\sigma_{k+1}}^{(k+1)}$ are in $H^+(x_0)$. But
since $0 \in \conv\{y_{\sigma_1}^{(1)}, \dots, y_{\sigma_{k+1}}^{(k+1)}\}$ this implies that the convex hull of the points in 
$\{y_{\sigma_1}^{(1)}, \dots, y_{\sigma_{k+1}}^{(k+1)}\} \cap H(x_0)$ contains the origin. Now, the dimension of $H(x_0)$ is $k-1$, and thus by Carath\'eodory's theorem there exists a set $S$ of 
at most $k$ of the points in $y_{\sigma_1}^{(1)}, \dots, y_{\sigma_{k+1}}^{(k+1)}$ with $0\in \conv S$, in contradiction to our general position assumption.

We have shown that $\bigcap_{i=1}^{k+1} A_{\sigma_i}^{(i)} = \{0\}$, and thus in particular, $A_{\sigma_i}^{(i)} \ne \sigma_i$ for all $i$. By our definitions, this implies $y_{\sigma_i}^{(i)} \in X_i$ for all $i$, concluding the proof of the theorem.  
\end{proof} 

\begin{remark}
Note that we could have avoided the usage of Carath\'eodory's theorem in the proof of Theorem \ref{thm:col-car} by taking a more restrictive assumption on the polytope $P$, namely, that ${\langle x,y \rangle > 0}$ whenever the points $x$ and $y$ belong to the same face of $P$. Therefore, in particular, Theorem \ref{thm:col-car} specializes to Carath\'eodory's theorem in the case where all the sets $X_i$ are the same.  
\end{remark}

\section{A colorful $d$-interval theorem}
\label{sec:interval}

Recall that 
a {\em fractional matching} in a hypergraph $H=(V,E)$ is a function $f\colon E \longrightarrow \R_{\ge 0}$ satisfying $\sum_{e:~ e\ni v} f(e)\le 1$ 
for all~${v\in V}$. A {\em fractional cover} is a function $g\colon V \longrightarrow \R_{\ge 0}$ satisfying $\sum_{v:~ v \in e} g(v)\ge 1$ 
for all~${e\in E}$. The \emph{fractional matching number} $\nu^*(H)$ is the maximum of $\sum_{e\in E} f(e)$ over all fractional matchings $f$ of $H$, and
the \emph{fractional covering number} $\tau^*(H)$ is the minimum of $\sum_{v\in V} g(v)$ over all fractional covers $g$. By linear programming duality, $\nu\le \nu^*=\tau^*\le \tau$.
A {\em perfect fractional matching} in $H$ is a fractional matching $f$ in which $\sum_{e:v\in e} f(e) = 1$ for every $v\in V$. 
It is a simple observation that a collection of sets $\I \subset 2^{[k+1]}$  is balanced if and only if the hypergraph $H=([k+1], \I)$ has a perfect fractional 
matching (see e.g., \cite{AKZ}).
The {\em rank} of a hypergraph $H=(V,E)$ is the maximal size of an edge in $H$.  $H$ is $d$-partite if there exists a partition $V_1,\dots, V_d$ of $V$ such that $|e\cap V_i| =1$ for every $e\in E$ and $i\in [d]$. 

For the proof of Theorem \ref{coloreddintervals} we will use the following theorem by F\"uredi.
\begin{theorem}[F\"uredi \cite{furedi}]\label{furedi}
If $H$ is a hypergraph of rank $d\ge 2$, then
$\nu(H) \ge \frac{\nu^*(H)}{d-1+\frac{1}{d}}.$ If, in addition, $H$ is $d$-partite, then $\nu(H) \ge \frac{\nu^*(H)}{d-1}.$  
\end{theorem}

We will also need the following simple counting argument. 
\begin{lemma}\label{rank}
If a hypergraph $H=(V,E)$ of rank $d$ has a perfect fractional matching, then $\nu^*(H)\ge \frac{|V|}{d}$.  
\end{lemma}
\begin{proof}
Let $f\colon E\longrightarrow \mathbb{R}_{\ge 0}$ be a perfect fractional matching of $H$. Then 
$\sum_{v\in V}\sum_{e: v\in e} f(e) = \sum_{v\in V} 1 = |V|.$
Since $f(e)$ was counted $|e|\le d$ times in this equation for every edge $e\in E$, we have that 
$\nu^*(H)\ge \sum_{e\in E} f(e) \ge \frac{|V|}{d}.$   
\end{proof}  

We are now ready to prove Theorem~\ref{coloreddintervals}. The proof is an adaption of the methods in~\cite{AKZ}.
For the first part we need the simplex version of Theorem~\ref{thm:col-komiya}, which was already proven by Shih and Lee~\cite{ShihLee},
while the second part requires our more general polytopal extension.

\begin{proof}[Proof of Theorem~\ref{coloreddintervals}.]
For a point 
$\vec{x}=(x_1,\dots,x_{k+1}) \in \Delta_k$ let  $p_{\vec{x}}({j})=\sum_{t=1}^j x_t  \in [0,1]$.
Since $\F$ is finite, by rescaling $\R$ we may assume that 
$\F \subset (0,1)$. 
For every $T\subset [k+1]$ let $A^i_T$ be the set consisting of all $\vec{x} \in
\Delta_k$ for which there exists a $d$-interval $f \in \F_i $ satisfying: 
\begin{enumerate}
\item[(a)] $f\subset \bigcup_{j\in T} (p_{\vec{x}}({j-1}),p_{\vec{x}}({j}))$, and 
\item[(b)] $f \cap
(p_{\vec{x}}({j-1}),p_{\vec{x}}({j})) \neq \emptyset$ for each $j
\in T$.
\end{enumerate}  
Note that $A^i_T =
\emptyset$ whenever $|T| > d$.

Clearly, the sets $A^i_{T}$ are open. The assumption $\tau(\F_i )>k$ implies that for every $\vec{x}=(x_1,\dots,x_{k+1}) \in \Delta_k$, the set $P(\vec{x})=\{p_{\vec{x}}({j}) \::\: j\in [k]\}$ is not a cover of $\F_i $,  
meaning that there exists $f\in \F_i $ not
containing any $p_{\vec{x}}(j)$. This, in turn, means that $\vec{x}
\in A^i_{T}$ for some $T \subseteq [k+1]$, and thus the sets $A^i_{T}~$ form a cover of $\Delta_k$ for every $i\in[k+1]$.
 
To show that this is a KKMS cover, let $\Delta^S$ be a face of $\Delta_k$ for some $S\subset [k+1]$. If $\vec{x}\in \Delta^S$ then
$(p_{\vec{x}}({j-1}),p_{\vec{x}}({j}))=\emptyset$ for $j \notin S$,
and hence it is impossible to have ${f \cap (p_{\vec{x}}({j-1}),p_{\vec{x}}({j}))\neq \emptyset}$. 
Thus  $\vec{x} \in A^i_{T}$
for some $T \subseteq S$. This proves that $\Delta^S \subseteq \bigcup_{T
\subseteq S}A^i_{T}$ for all $i\in [k+1]$. 

By Theorem~\ref{thm:col-komiya}
 there exists 
  a balanced collection of sets  $\T=\{T_1,\dots,T_{k+1}\}$ of subsets of $[k+1]$, satisfying
 $\bigcap_{i=1}^{k+1} A^i_{T_i}\neq \emptyset$. In particular,
$|T_i|\le d$ for all $i$. (Recall that we think of a collection of sets $\mathcal I \subset 2^{[k+1]}$ as faces of the $k$-dimensional simplex to apply the earlier geometric definition of balancedness.) Then by the observation mentioned above, the hypergraph $H=([k+1],\T)$ of rank $d$ has a perfect fractional matching, and thus by Lemma \ref{rank} we have $\nu^*(H) \ge
\frac{k+1}{d}$. Therefore, by Theorem \ref{furedi}, $\nu(H) \ge
\frac{\nu^*(H)}{d-1+\frac{1}{d}}\ge \frac{k+1}{d^2-d+1}.$

Let $M$ be a matching in $H$ of size  $m \ge \frac{k+1}{d^2-d+1}$.
Let $\vec{x} \in \bigcap_{i=1}^{k+1} A^i_{T_i}$. For every $i\in [k+1]$ let $f(T_i)$ be the $d$-interval of $\F_i $ witnessing the fact that
$\vec{x} \in A^i_{T_i}$. Then the set $\M=\{f(T_i)\mid T_i \in M\}$ is a matching of
size $m$ in $\F$ with $|\M\cap \F_i | \le 1$. This proves the first assertion of the theorem.

Now suppose that $\F_i $ is a hypergraph of separated $d$-intervals for all $i\in [kd+1]$. 
For $f \in \F$ let $f^t \subset (t-1,t)$ be the $t$-th interval component of~$f$.
We can assume without loss of generality that $f^t$ is nonempty.  
Let $P=(\Delta_k)^d$. For a $d$-tuple $T=(j_i,\dots,j_d) \subset [k+1]^d$ let $A^i_T$ consist of all $\vec{X}=\vec{x}^1\times\cdots\times\vec{x}^d \in
P$ for which there exists $f \in \F_i $ satisfying 
 $f^t\subset (t-1+p_{\vec{x}^t}({j_t-1}),t-1+p_{\vec{x}^t}({j_t}))$ for all $t\in[d]$.

Since $\tau(\F_{i})>kd$, the points $t-1+p_{\vec{x}^t}(j), t\in[d], j\in [k],$ do not form a cover of $\F_{i}$. Therefore,  
by the same argument as before, the sets $A^i_{T}$ are open and satisfy the covering condition of Theorem \ref{thm:col-komiya}. Thus, 
by Theorem \ref{thm:col-komiya},
 there exists 
  a set  $\T=\{T_1,\dots,T_{kd+1}\}$  of $d$-tuples in $[k+1]^d$ containing the point $(\frac{1}{k+1},\dots,\frac{1}{k+1})\times\dots \times(\frac{1}{k+1},\dots,\frac{1}{k+1})\in P$ in its convex hull and satisfying $\bigcap_{i\in [kd+1]} A^i_{T_i} \neq \emptyset$. Then the $d$-partite hypergraph  $H=(\bigcup_{i=1}^d V_i,\T)$, where $V_i=[k+1]$ for all $i$, has a  perfect fractional matching, and hence by Lemma \ref{rank} we have $\nu^*(H) \ge k+1$. By Theorem \ref{furedi}, this implies $\nu(H) \ge
\frac{\nu^*(H)}{d-1}\ge \frac{k+1}{d-1}
.$ Now, by the same argument as before, by taking $\vec{X} \in \bigcap_{i\in [kd+1]} A^i_{T_i}$ we obtain a matching in $\F$ of the same size as a maximal matching in $H$, concluding the proof of the theorem.
\end{proof}

\section*{Acknowledgment}
We wish to thank Erel Segal-Halevi for pointing out a mistake in the statement of Theorem 1.3(2) in a previous version of this manuscript. 
This material is based upon work supported by the National Science Foundation under Grant No. DMS-1440140 while the authors were  in residence at the Mathematical Sciences Research Institute in Berkeley, California, during the Fall 2017 semester.

\bibliographystyle{amsplain}

\end{document}